\documentclass[11pt]{amsart}

\usepackage{amsmath, amscd, amssymb}
\usepackage[frame,cmtip,arrow,matrix,line,graph,curve]{xy}
\usepackage{graphpap, color}
\usepackage[mathscr]{eucal}
\usepackage{cancel}
\usepackage{verbatim}
 
\numberwithin{equation}{section}

\def\sO{{\mathscr O}}

\newcommand{\CC}{\mathbb{C}}

\newcommand{\PP}{\mathbb{P}}

\newcommand{\RR}{\mathbb{R}}
\newcommand{\ZZ}{\mathbb{Z}}

\newcommand{\cal}{\mathcal}

\def\cC{{\cal C}}

\def\cH{{\cal H}}

\def\cL{{\cal L}}

\def\and{\quad{\rm and}\quad}
\def\lra{\longrightarrow }
\def\mapright#1{\,\smash{\mathop{\lra}\limits^{#1}}\,}
\def\mapleft#1{\,\smash{\mathop{\longleftarrow}\limits^{#1}}\,}

\let\sub=\subset

\newtheorem{prop}{Proposition}[section]
\newtheorem{theo}[prop]{Theorem}

\newtheorem{coro}[prop]{Corollary}

\newtheorem{defi}[prop]{Definition}

\newtheorem{assu}[prop]{Assumption}

\def\beq{\begin{equation}}
\def\eeq{\end{equation}}

\def\virt{^{\mathrm{vir}} }

\def\DM{Deligne-Mumford }

\def\zero{\mathrm{zero} }

\def\bl{\bigl(}
\def\br{\bigr)}
\def\redd{{\mathrm{red}}}

\def\fX{\mathfrak{X} }

\def\tE{\widetilde{E} }

\def\tX{\widetilde{X} }

\def\bp{\mathbf{p}} 
\def\bq{\mathbf{q}}

\def\tX{\widetilde{X} }

\def\tS{\widetilde{S}}
\def\tE{\widetilde{E}}
\def\cC{\mathcal{C}}
\def\zero{\mathrm{zero}}

\def\tsig{\tilde{\sigma}}
\def\tY{\widetilde{Y}}
\def\DM{Deligne-Mumford }
\def\ih{I\!H}
\def\ulambda{{\lambda}}
\def\ulog{^{\mathrm{log}}}
\def\urig{^{\mathrm{rig}}}
\def\ugamma{{\underline{\gamma}}}
\def\rmgn{\overline{M}_{g,n}}
\def\fZ{{\mathfrak{Z}}}
\def\tfZ{{\widetilde{\fZ}}}
\def\lloc{_{\mathrm{loc}}}

\title{Algebraic virtual cycles for quantum singularity theories}
\date{2018.10.30.}
\author{Huai-Liang Chang}
\address{Department of Mathematics, Hong Kong University of Science and Technology, Hong Kong}
\email{mahlchang@ust.hk}
\author{Young-Hoon Kiem}
\address{Department of Mathematics and Research Institute
of Mathematics, Seoul National University, Seoul 08826, Korea}
\email{kiem@snu.ac.kr}
\author{Jun Li}
\address{Department of Mathematics,  Stanford University, CA 94305, USA}
\email{jli@stanford.edu}

\thanks{YHK was partially supported by Samsung Science and Technology Foundation SSTF-BA1601-01; JL was partially supported by NSF grants DMS-1564500 and DMS-1601211.}

\begin{document}
\begin{abstract} We construct algebraic virtual cycles that give us the cohomological field theories of Fan-Jarvis-Ruan invariants by integral transformations.\end{abstract}
\maketitle

\section{Introduction}\label{S1}

In this paper, we construct an algebraic virtual cycle that provides us with the cohomological field theory of Fan-Jarvis-Ruan-Witten (FJRW for short) invariants in \cite{FJR} by a Fourier-Mukai type integral transformation. 

\subsection{Background and motivation}
Let $w:\CC^N\to \CC$ be a nondegenerate quasi-homogeneous polynomial (cf. \S\ref{S5.1}) which defines a nonsingular hypersurface $Q_w=\PP w^{-1}(0)$. Let $\hat{G}$ be a subgroup of $(\CC^*)^N$ and $\chi:\hat{G}\to \CC^*$ be a homomorphism such that $w(g\cdot x)=\chi(g) w(x)$. The kernel of $\chi$ is a finite group denoted by $G$. If we let $\hat{G}$ act on $\CC^N\times \CC$ by
 $g\cdot (x,t)=(g\cdot x,\chi(g)^{-1}t)$, the quotient stack 
$$\fX=[(\CC^N\times \CC)/\hat{G}]$$
admits a function $\mathbf{w}(x,t)=t\cdot w(x)$ and two GIT quotients
$$\fX_+=\bl(\CC^N-0)\times\CC\br/\hat{G},\quad \fX_-=\bl\CC^N\times (\CC-0)\br/\hat{G}=\CC^N/G.$$ 
The former $\fX_+$ is an (orbi-)line bundle over the weighted projective space
$\PP^{N-1}$ and the critical locus of $\mathbf{w}|_{\fX_+}$ is $Q_w$, up to quotient by a finite group $G/\mu_d$. 
On the other hand, $\mathbf{w}|_{\fX_-}=w$.  

In \cite{Witt}, Witten conjectured that the Gromov-Witten invariants of $Q_w$ 
should be computable by the Landau-Ginzburg (LG for short) model 
$$w:\CC^N/G\lra \CC$$
whose curve counting invariants 
should be integrals on the solution space of Witten's equation on the moduli space $X$ 
of $G$-spin curves $(C,p_j, L_i, \varphi)$ (cf. \S\ref{S5.2})
together with sections $(x_i)\in \prod_{i=1}^NH^0(L_i)$.

Through analysis, Fan, Jarvis and Ruan in \cite{FJR} studied the solution space of Witten's equation and defined the 
FJRW invariants which were proved to satisfy nice properties like the splitting axioms, codified as cohomological field theories.  Quantum singularity theories in the title refer to cohomological field theories arising from singularities like $w^{-1}(0)$. 

Slightly later, Polishchuk and Vaintrob in \cite{PV} provided a purely algebraic construction of cohomological field theories of the LG model $w:\CC^N/G\to \CC$ by matrix factorizations.
 They constructed a universal matrix factorization and their cohomological field theories are obtained by
 Fourier-Mukai type transformations on matrix factorizations and Hochschild homology. 
As the functors of matrix factorizations do not preserve the ordinary cohomology degrees, the algebraic theory in \cite{PV} lacks in explicit interpretation in terms of cycles and basic properties like 
the homogeneity of dimension are not obvious. 

An algebraic theory for FJRW invariants by algebraic cycles was provided in \cite{CLL} 
for narrow sectors by constructing the virtual fundamental cycle for the moduli space $X$ 
where Witten's equation is replaced by the cosection localization principle (cf. \cite{KLc}).
For the general case including broad sectors,  the second and third named authors in \cite{KLq}  
generalized the cosection localization of \cite{KLc} to intersection homology and 
provided a direct construction of the cohomological field theories for both broad and narrow sectors. 
As the construction in \cite{KLq} does not involve virtual cycles,
one may wonder whether it is possible to construct the cohomological field theories by a Fourier-Mukai type integral transformation whose kernel is an algebraic virtual cycle. 

The goal of this paper is to construct algebraic virtual cycles that give us the cohomological field theories of \cite{KLq} by integral transformations.

\subsection{Construction of virtual cycles by blowups}

The moduli stack $X$ of rigidified $G$-spin curves with sections 
can be written as the zero locus of a section $s$ of a vector bundle $E$ 
over a \DM stack $Y$ (cf. \S\ref{S5.3}). 
We also have a cosection
$\sigma:E\to \sO_Y$ satisfying $\sigma\circ s=0$ and 
a smooth morphism $\bq:Y\to Z=\underline{w}^{-1}(0)$, where 
$\underline{w}$ is a nondegenerate quasi-homogeneous polynomial on a finite dimensional vector space.
As $X$ is usually not proper, the ordinary virtual fundamental class $[X]\virt$, as a Chow cycle supported in $X$, 
cannot be used for an integral transformation. 
On the other hand, the intersection $S$ of $X$ and the degeneracy locus $\fZ$ 
(sometimes called the zero locus) of 
$\sigma$ is the moduli space of rigidified $G$-spin curves; hence $S=X\cap \fZ$ is proper. 

A simple observation shows that when $Y$ is smooth, the cosection $\sigma$ descends 
to a cosection of the obstruction sheaf $Ob_X=\mathrm{coker}(ds)$ of the perfect obstruction theory $[T_Y|_X\mapright{ds} E|_X]$.
In the narrow case, $Z=0$ and one can apply the cosection localization principle in \cite[Theorem 5.1]{KLc}
to obtain the cosection localized virtual cycle 
$$[X]\virt\lloc\in A_*(S)$$
which gives us the FJRW invariants in the narrow case (cf. \cite{CLL}). 

In the broad case, $Y$ is singular and $\sigma$ does not descend to $Ob_X$. 
In order to apply the cosection localization in \cite{KLc}, 
we will replace $Z$ by its blowup $Z'$ at the isolated singular point $0$, and
pull back all the data above to $Z'$ to get a smooth morphism $\bq':Y'\to Z'$, 
a vector bundle $E'$ over $Y'$, a section $s'$ of $E'$ which defines $X'$, and 
a cosection $\sigma'$ of $E'$ whose zero locus $\fZ'$ intersects with $X'$ along the proper $S'=Z'\times_ZS$. 
Since $Z'$ is smooth and $\sigma'\circ s'=0$, $Y'$ is smooth and the cosection $\sigma'$ descends
to a cosection of the obstruction sheaf $Ob_{X'}=\mathrm{coker}(ds')$.
Applying \cite[Theorem 5.1]{KLc}, we obtain a cosection localized virtual cycle
\[ [X']\virt\lloc\in A_*(S').\]

The proper morphism $p$ and the composite $q$ below
$$p:S'\lra S,\qquad q:S'\hookrightarrow Y'\mapright{\bq'} Z',
$$
where the former is induced from the blowup morphism $Z'\to Z$,
together with the virtual cycle $[X']\virt\lloc$, give rise to an integral transformation
\beq\label{g52}
\Phi_{[X']\virt\lloc}:H^*(Z') \lra H_*(S),\quad \alpha\mapsto p_*([X']\virt\lloc\cap q^*\alpha).\eeq
The insertion space $\cH^{\otimes n}$ of the FJRW theory (cf. \S\ref{S5.1}) is contained in the direct sum of spaces of the form 
\[ \ih_m(Z)\subset H^{m-2}(Z')\cong H_m(Z')\]
and we have a proper pushforward $H_*(S)\to H_*(\rmgn)$. 
Hence \eqref{g52} enables us to define homomorphisms
\beq\label{g54} \Omega'_{g,n}:\cH^{\otimes n}\lra H_*(\rmgn)\cong H^*(\rmgn).\eeq
 
 In \cite[Theorem 3.2]{KLq}, the second and third named authors constructed the cosection localized Gysin maps for intersection homology
 \beq\label{g55} s^!_\sigma:\ih_*(Y)\lra H_*(S).\eeq
Composed with the pullback $\bq^*:\ih_*(Z)\to \ih_*(Y)$, \eqref{g55} also enables us to define homomorphisms
\beq\label{g56} \Omega_{g,n}:\cH^{\otimes n}\lra H_*(\rmgn)\cong H^*(\rmgn).\eeq
In \cite[Theorem 4.5]{KLq}, it was proved that \eqref{g56} satisfies the axioms of the FJRW cohomological field theory. 

The goal of this paper is to prove that 
\[ \Omega'_{g,n}=\Omega_{g,n} ,
\]
and hence the integral transformation \eqref{g52} by the algebraic virtual cycle $[X']\virt\lloc$ provides us with the FJRW cohomological field theory \eqref{g54} (cf. Theorem \ref{g30}).

\subsection{Layout} 

In \S\ref{S5}, we recall the spin curves and their moduli space. In \S\ref{S2}, we construct virtual cycles by blowup and define an integral transformation whose kernel is the virtual cycle. In \S\ref{S3}, we recall the cohomological field theory construction by intersection homology in \cite{KLq}.
In \S\ref{S4}, we prove the main theorem about the equality of the two cohomological field theories constructed in the previous sections.

\subsection{Notation and convention}

All varieties, schemes and stacks are defined over $\CC$ in this paper. We will use only the classical topology of algebraic varieties and schemes.
All the topological spaces in this paper are locally compact Hausdorff countable CW complexes.  
Intersection homology in this paper refers to the middle perversity intersection homology unless stated otherwise. The Borel-Moore homology groups are denoted by $H_*(-)$. We will not use the ordinary homology groups.  
All the cohomology groups in this paper have complex coefficients. 
The fundamental class of an irreducible closed substack $V$ of a \DM stack $Y$  in the Chow group $A_*(Y)$ is denoted by 
$[\![V]\!]$ while the fundamental class of $V$ in the Borel-Moore homology group $H_*(Y)$ is denoted by $[V]$.

\section{Spin curves and curve counting}\label{S5}

In this section, we recall the Fan-Jarvis-Ruan-Witten theory from \cite{FJR}. Our presentation follows \cite[\S4]{KLq}.

\subsection{Hypersurface singularities}\label{S5.1}

A polynomial $w:\CC^N\to \CC$ is quasi-homogeneous if for some $d_1,\cdots,d_N, d\in \ZZ_{>0}$,
\beq\label{d7}
w(t^{d_1}x_1,\cdots,t^{d_N}x_N)=t^d \cdot w(x_1,\cdots,x_N).
\eeq 
Here we assume that $d>0$ is the minimal possible. Let $q_i=d_i/d$. 
The quasi-homogeneous polynomial $w$ is \emph{nondegenerate} if the following are satisfied: 
\begin{enumerate}
\item no mominial of $w$ is of the form $x_ix_j$ for $i\ne j$;
\item the projective hypersurface $Q_w$ defined by $w$ is nonsingular:
$$Q_w=\PP w^{-1}(0)\subset \PP^{N-1}_{d_1,\cdots,d_N}.$$
\end{enumerate} 
By (2), the hypersurface $w^{-1}(0)\subset \CC^N$ has singularity only at the origin $0$ and $q_i\le \frac12$. 

We write $w=\sum_{k=1}^\nu c_kw_k$, where $c_k\in \CC^*$ and $w_k$ are distinct monomials. 
The kernel of the homomorphism
\beq\label{d25} (w_1,\cdots, w_\nu): (\CC^*)^N\lra (\CC^*)^\nu \eeq
is the symmetry group
\beq\label{d8} G_w=\{(\lambda_1\cdots,\lambda_N)\in (\CC^*)^N\,|\,w(\lambda_1 x_1,\cdots,\lambda_N x_N)=w(x_1,\cdots,x_N)\}\eeq
of $w$, which is finite by the nondegeneracy. 
Let 
$$J_w=(e^{2\pi iq_1},\cdots, e^{2\pi iq_N})\in G_w,
$$ 
and fix a subgroup $G$ of $G_w$ containing $J_w$. The pair $(w,G)$ is the input data for the FJRW theory in \cite{FJR}.

Consider the the diagonal embedding $\CC^*\to (\CC^*)^\nu$ and the fiber product 
\[\xymatrix{
\hat{G}_w\ar[r] \ar[d] &\CC^*\ar[d]\\
(\CC^*)^N\ar[r] & (\CC^*)^\nu
}\]
of \eqref{d25}. 
By the quasi-homogeneity \eqref{d7}, the homomorphism $$\CC^*\to  (\CC^*)^N,\quad t\mapsto (t^{d_1},\cdots,t^{d_N})$$
factors through $\hat{G}_w$, which together with the inclusion $G_w\to \hat{G}_w$  
gives us a surjective homomorphism 
$G_w\times \CC^*\to \hat{G}_w$ whose kernel is $\mu_d\le \CC^*$, the group of $d$-th roots of unity.
The subgroup $G$ of $G_w$ thus determines  
$$\hat{G}=G\times\CC^*/\mu_d \subset\hat{G}_w$$ that fits into an exact sequence 
\beq\label{d26} 1\lra G \lra \hat{G} \mapright{\chi} \CC^*\lra 1.\eeq
Since $G$ acts trivially on $w$, for ${\lambda}\in \hat{G}\subset (\CC^*)^N$, 
\beq\label{d31} w(\ulambda\cdot x)=\chi(\ulambda)w(x).\eeq 

The \emph{state space} for the singularity $(w,G)$ in \cite{FJR} is 
\beq\label{d9} \cH=\bigoplus_{\gamma\in G}\cH_\gamma,\quad \cH_\gamma=H^{N_\gamma}(\CC^{N_\gamma},w_\gamma^\infty)^G.
\eeq
Here $\CC^{N_\gamma}$ is the $\gamma$-fixed subspace of $\CC^N$, and $w_\gamma^\infty=(\mathrm{Re}(w_\gamma))^{-1}(a,\infty)$, where $w_\gamma=w|_{\CC^{N_\gamma}}$ and $a>\!> 0$. 

If we let 
$$Z'_\gamma\lra Z_\gamma=w_\gamma^{-1}(0)
$$ 
be the weighted blowup at the origin, then $Z'_\gamma$ is the line bundle $\sO_{Q_{w_\gamma}}(-1)$ which is the restriction of $\sO_{\PP^{N_\gamma-1}_{d_1,\cdots,d_{N_\gamma}}}(-1)$ to $Q_{w_\gamma}=\PP w_\gamma^{-1}(0)\subset \PP^{N_\gamma-1}_{d_1,\cdots,d_{N_\gamma}}$. As $Q_{w_\gamma}$ is smooth by \cite[Lemma 2.1.10]{FJR}, so is $Z'_\gamma$.

We consider the vanishing cohomology $H^{*}_{\mathrm{van}}(\PP w_\gamma^{-1}(0))$ and the primitive cohomology 
$H^{*}_{\mathrm{prim}}(\PP w_\gamma^{-1}(0))$ of $Q_w=\PP w_\gamma^{-1}(0).$
Complex Morse theory then provides us with isomorphisms
$$H^{N_\gamma}(\CC^{N_\gamma},w_\gamma^\infty)\cong H^{N_\gamma-2}_{\mathrm{van}}(\PP w_\gamma^{-1}(0))\cong H^{N_\gamma-2}_{\mathrm{prim}}(\PP w_\gamma^{-1}(0)),
$$ 
by \cite[Proposition 2.27]{Voisin} 
since the weighted projective space has no primitive cohomology in non-zero degrees. 
On the other hand, the middle perversity intersection homology of $w_\gamma^{-1}(0)$ satisfies
(cf. \cite[p.20]{Borel})
\beq\label{d11}\ih_i(w_\gamma^{-1}(0))=\begin{cases}
H^{N_\gamma-2}_{\mathrm{prim}}(\PP w_\gamma^{-1}(0)), & i=N_\gamma\\
\CC, & i=2N_\gamma-2 \\
0, & \text{otherwise}.
\end{cases}
\eeq 
Hence we have 
$$\cH_\gamma=\ih_{N_\gamma}(w_\gamma^{-1}(0))^G
\subset H^{N_\gamma-2}_{\mathrm{prim}}(\PP w_\gamma^{-1}(0))\subset H^{N_\gamma-2}(\PP w_\gamma^{-1}(0))
=H^{N_\gamma-2}(Z'_\gamma).
$$

For $\gamma_1,\gamma_2\in G$, the Thom-Sebastiani sum 
$$w_{\gamma_1}\boxplus w_{\gamma_2}: \CC^{N_{\gamma_1}}\oplus \CC^{N_{\gamma_2}}\to \CC
$$ 
is defined by $(x,y)\mapsto w_{\gamma_1}(x)+w_{\gamma_2}(y).$  
By \cite{Massey}, we have canonical isomorphisms and a commutative square
$$\xymatrix{
H^{N_{\gamma_1}-2}_{\mathrm{van}}(\PP w_{\gamma_1}^{-1}(0))\otimes H^{N_{\gamma_2}-2}_{\mathrm{van}}(\PP w_{\gamma_2}^{-1}(0))\ar[d]^\cong \ar[r]^\cong & H^{N_{\gamma_1}+N_{\gamma_2}-2}_{\mathrm{van}}(\PP (w_{\gamma_1}\boxplus w_{\gamma_2})^{-1}(0))\ar[d]^\cong\\
\ih_{N_{\gamma_1}}(w_{\gamma_1}^{-1}(0))\otimes \ih_{N_{\gamma_2}}(w_{\gamma_2}^{-1}(0))\ar[r]^\cong & \ih_{N_{\gamma_1}+N_{\gamma_2}}((w_{\gamma_1}\boxplus w_{\gamma_2})^{-1}(0)).
}
$$
Therefore, $\cH_{\gamma_1} \otimes \cH_{\gamma_2}$ is canonically isomorphic to
\beq\label{d19}  \ih_{N_{\gamma_1}+N_{\gamma_2}}((w_{\gamma_1}\boxplus w_{\gamma_2})^{-1}(0))^{G\times G}\cong H^{N_{\gamma_1}+N_{\gamma_2}-2}_{\mathrm{van}}(\PP (w_{\gamma_1}\boxplus w_{\gamma_2})^{-1}(0))^{G\times G}.\eeq

\subsection{Moduli of spin curves}\label{S5.2}

Given the input data $(w,G)$, we have the moduli stack of spin curves.  

A pointed \emph{twisted curve} refers to a proper \DM stack $C$ with smooth substacks $p_1,\cdots, p_n\subset C$, such that 
\begin{enumerate}
\item denoting the coarse moduli space by $\rho:C\to |C|$, $|C|$ is a projective curve which has at worst nodal singularities and the markings $\rho(p_i)=|p_j|$ are smooth points of $|C|$;
\item $\rho$ is an isomorphism away from special points (nodes or makings);
\item a marking is locally $\CC/\mu_{l}$ for some $l>0$, where $\mu_l$ is the group of $l$-th roots of unity; 
\item a node is locally $\{xy=0\}/\mu_l$ for some $l>0$, where $\mu_l$ acts via $(x,y)^z=(zx,z^{-1}y)$.
\end{enumerate}
The log dualizing sheaf of $C$ is the pullback $$\omega_C\ulog=\rho^*\omega_{|C|}\ulog=\rho^*\omega_{|C|}(|p_1|+\cdots+|p_n|).$$

A \emph{$G$-spin curve} is a principal $\hat{G}$-bundle $P$ on a pointed twisted curve $(C,p_1,\cdots, p_n)$  equipped with an isomorphism
\[ \varphi:\chi_*P\cong P(\omega_C\ulog)\] 
of principal $\CC^*$-bundles. Here $P(\omega_C\ulog)$ is the principal  $\CC^*$-bundle associated to the line bundle 
$\omega_C\ulog$; $\chi$ is as in \eqref{d26} and $\chi_*P$ is the principal $\CC^*$-bundle obtained by applying $\chi$ to the fibers of $P$. 
Applying the inclusion map $\hat{G}\to (\CC^*)^N$ to $P$, we obtain a  principal $(\CC^*)^N$-bundle $P\times_{\hat{G}}(\CC^*)^N$ over $C$ which gives us line bundles $(L_1,\cdots, L_N)$. 
The stabilizer group $G_{p_j}$ of a marking $p_j$ acts on the fiber $\oplus_i L_i|_{p_j}$ by 
$\gamma_j=(\gamma_{ij})_{1\le i\le N}\in G$. We let $\ugamma=(\gamma_1,\cdots,\gamma_n)\in G^n$ and call it the type of the $G$-spin curve. 
The stabilizer group $G_p$ of a node $p$ acts on $\oplus_i L_i|_p$ by a $\gamma_p\in G$. 
  
The spin curve $(C,p_j, L_i,\varphi)$ is \emph{stable} if $(|C|,|p_1|,\cdots, |p_n|)$ 
is a stable curve, and the homomorphism from the stabilizer group of a marking $p_j$ (resp. a node $p$) into $G$ that sends the generator to $\gamma_j$ (resp. $\gamma_p$) is injective.

\begin{theo}\label{d29}\cite[Theorem 2.2.6]{FJR} \cite[Proposition 3.2.6]{PV}
The stack $S_{g,n}$ of stable $G$-spin curves is a smooth proper \DM stack with projective coarse moduli.
The forgetful morphism $S_{g,n}\to \overline{M}_{g,n}$ sending a $G$-spin curve $(C,p_j,L_i,\varphi)$ to the underlying stable curve $(|C|, |p_j|)$ is flat proper and quasi-finite.  
\end{theo}

A \emph{rigidification} of a $G$-spin curve at a marking $p_j$ is an isomorphism 
$$\psi_j:L_1|_{p_j}\oplus \cdots \oplus L_N|_{p_j}\mapright{\cong} \CC^N/\langle\gamma_j\rangle,
$$ 
where $\langle\gamma_j\rangle\le G$ is the subgroup generated by $\gamma_j$, such that 
$w_k\circ \psi_j=\mathrm{res}_{p_j}\circ \varphi_k|_{p_j}$ for every monomial $w_k$ of $w$. 
The moduli stack $S_{g,n}\urig$ of stable $G$-spin curves with rigidification is an \'etale cover over $S_{g,n}$ and hence $S_{g,n}\urig$ is a proper smooth \DM stack. 
The moduli stack of $G$-spin curves of type $\ugamma$ with rigidification is denoted by $S\urig_{g,\ugamma}$.  So we have the disjoint union $$S\urig_{g,n}=\sqcup_\ugamma S\urig_{g,\ugamma}.$$

\subsection{Moduli of spin curves with sections}\label{S5.3} 
To simplify the notation, let $S=S\urig_{g,\ugamma}.$ 
Let $\cL_i$ be the universal line bundle over the universal curve $\pi:\cC\to S$ over $S$. 
By \cite[\S4.2]{PV}, there are a locally free resolution 
\beq\label{d34} R\pi_*(\oplus_{i=1}^N \cL_i)\cong [M\mapright{\beta} F]\eeq
and a smooth morphism 
$$\bq_M:M\lra B=B_\ugamma=\prod_{j=1}^n\CC^{N_{\gamma_j}}.$$
The Thom-Sebastiani sum  
${w}_{\ugamma} =w_{\gamma_1}\boxplus \cdots \boxplus w_{\gamma_n}$
is a polynomial function on $B=B_\ugamma$ whose zero locus is denoted by $Z=Z_\ugamma$.

Let $E_M=\bp_M^*F$ and $s_M$ be the section of $E_M$ defined by $\beta$ where $\bp_M:M\to S$ 
is the bundle projection. Then the zero locus $s_M^{-1}(0)$ of the section is the moduli space 
$$X=X\urig_{g,\ugamma}$$ of stable $G$-spin curves $(C,p_j, L_i,\varphi)$ of type $\ugamma$ 
together with rigidification $\psi$ and sections $(x_1,\cdots,x_N)\in \oplus_i H^0(L_i)$ of the 
line bundles $L_i$. 
By \cite[\S4.2]{PV}, $E_M$ admits a cosection (i.e. a homomorphism to the structure sheaf) $\sigma_M:E_M\to \sO_M$, 
which satisfies 
\[ \sigma_M\circ s_M=\underline{w} \circ\bq_M\and X\cap\sigma_M^{-1}(0)_\redd=S.\]
Since the sum of residues of any meromorphic 1-form over a curve is zero, 
\beq\label{a7} X=X\urig_{g,\ugamma} \subset Y=Y_{g,\ugamma}:=Z\times_{B}M\subset M.\eeq
In summary we have the following diagram:
\beq\label{a3}
\xymatrix{
&& E_M\ar[d]\ar[r]^{\sigma_M} & \sO_M\ar@{=}[r] & M\times\CC\ar[r] & \CC\\
X\ar@{=}[r] & s^{-1}_M(0)\ar[dr]_{\bp_X}\ar@{^(->}[r] & M\ar@/^1.0pc/[u]^{s_M}\ar[urrr]_{\underline{w} \circ\bq_M}\ar[dr]^{\bq_M}\ar[d]^{\bp_M}\\
&&S&B\ar[r]^{\underline{w} } &\CC
}\eeq
By \eqref{a7}, we have a fiber diagram
\beq\label{d1}\xymatrix{
X\ar[dr]_{\bq_X}\ar@{^(->}[r]^\imath & Y\ar[r] \ar[d]^{\bq_Y} &M\ar[d]^{\bq_M}\\
&Z \ar@{^(->}[r] & B.
}\eeq
Here $\bq_Y$ is smooth as $\bq_M$ is smooth. 
The restriction of $E_M$ (resp. $\sigma_M$, resp. $s_M$) to $Y$ is denoted by 
$E$ (resp. $\sigma$, resp. $s$). By \eqref{a7},  $X=s^{-1}(0)$.

Because $Z$ has at most an isolated hypersurface singularity by our assumption on 
the quasi-homogeneous polynomial $w$, for $\dim_\CC Z=m-1\ge 2$, the affine variety $Z$ is 
normal and hence $Y$ is a normal as well. 
When $m\le 2$, we may replace $Z$ by its normalization. 
The intersection homology remains the same under normalization and all the arguments 
in this paper go through. 
Therefore for the FJRW theory, it suffices to work under the following.
\begin{assu}\label{aa}
Let $Y$ be a normal \DM stack over $\CC$. Let $s\in H^0(E)$ for a vector bundle $E$ of rank $r$ over $Y$ and let $X=s^{-1}(0)$. Let $\sigma\in H^0(E^\vee)=\mathrm{Hom}_Y(E,\sO_Y)$ be a cosection of $E$ satisfying  
\beq\label{a15} \sigma\circ s=0.\eeq
Let $\fZ=\sigma^{-1}(0)=\zero(\sigma)$ be the degeneracy locus where $\sigma$ is zero (i.e. not surjective).
We assume that $S=X\cap \fZ$ is proper and there is a smooth morphisms 
$\bq=\bq_Y:Y\to Z$ 
where $Z=w_\ugamma^{-1}(0)\subset B=\CC^m$ is the hypersurface defined by a nondegenerate quasi-homogeneous polynomial $w_{\ugamma}$. Let $g:Z'\to Z$ be the blowup of $Z$ at the origin, so that $Z'$ is smooth. We let $f:Y'=Y\times_ZZ'\to Y$ denote the pullback of $g$ by $\bq:Y\to Z$. 
\end{assu}
More precisely, $X$ (resp. $\fZ$) is the closed substack defined by the image of $s^\vee:E^\vee\to \sO_Y$ (resp. $\sigma:E\to \sO_Y$) in $\sO_Y$.

\section{Algebraic virtual cycles for the FJRW theory}\label{S2}

In this section, we construct algebraic virtual cycles for the Fan-Jarvis-Ruan-Witten theory and define Fourier-Mukai type integral transformations which will give us cohomological field theories in the subsequent sections.

We use the notation in \S\ref{S5}. 
Let $E, Y, s, \sigma$ be as in Assumption \ref{aa}. In particular, the moduli space $X=X_{g,\ugamma}\urig$ of rigidified $G$-spin curves with sections
$$(C,p_j,L_i, \varphi,\psi, x_i),
$$
where $(C,p_j,L_i, \varphi,\psi)\in S=S\urig_{g,\ugamma}$ and $x_i\in H^0(L_i)$,
is the zero locus of $s\in H^0(E)$ in $Y$. The restriction of the smooth morphism $\bq=\bq_Y$ to $X$ is denoted by  $\bq_X:X\to Z=Z_\ugamma$.

\subsection{Cosection localized virtual cycle by blowup}\label{S2.1}

By \cite[\S6]{BeFa}, there is a relative perfect obstruction theory
\beq\label{g1} \phi_{X/Z}: \mathbb{E}_{X/Z}=[E|_X^\vee\mapright{ds} \Omega_{Y/Z}|_X]\lra \mathbb{L}_{X/Z}\eeq
where $\mathbb{L}_{X/Z}=\tau^{\ge -1}L_{X/Z}$ is the truncated relative cotangent complex of $\bq_X$. 
In other words, $h^0(\phi_{X/Z})$ is an isomorphism and $h^{-1}(\phi_{X/Z})$ is surjective. 
Since $\sigma\circ s=w_\ugamma\circ \bq=0$ (cf. \eqref{a3}), $\sigma\circ ds|_{T_{Y/Z}}=0$ and hence $\sigma:E\to \sO_Y$ induces a cosection
\beq\label{g4} \sigma_{X/Z}:Ob_{X/Z}=\mathrm{coker}(T_{Y/Z}|_X\mapright{ds} E|_X)\lra \sO_X \eeq
of the relative obstruction sheaf. Let $X^\circ\sub X$ be the preimage of the smooth part $Z_{\text{sm}}\sub Z$. Then an easy argument shows that the cosection \eqref{g4} descends to a cosection of the obstruction sheaf $Ob_X|_{X^\circ}$ of the induced absolute perfect obstruction theory of $X^\circ$. The desired descent fails over $X-X^\circ$.


In order to obtain an absolute perfect obstruction theory with a cosection of its obstruction sheaf, we consider the blowup $g:Z'\to Z=Z_\ugamma$ at the origin and the fiber product
\beq\label{g11}\xymatrix{
S'\ar[r]^{\imath'}\ar[d]_p & X'\ar[r] \ar[d] & Y'\ar[r]^{\bq'}\ar[d]_{f}  & Z'\ar[d]^g\\
S\ar[r]^\imath & X\ar[r] & Y\ar[r]^\bq  & Z
}\eeq
so that $Y'$ (resp. $Z'$) is a smooth model of $Y$ (resp. $Z$).
We denote the pullbacks of $E, s, \sigma, \fZ$ to $Y'$ by $E', s', \sigma', \fZ$ so that $X'={\sigma'}^{-1}(0)$
and $\fZ'={\sigma'}^{-1}(0)$ while $S'=X'\cap \fZ'.$

The pullback of \eqref{g1} to $Y'$ is a relative perfect obstruction theory 
\beq\label{g2}
\phi_{X'/Z'}: \mathbb{E}_{X'/Z'}=[E'|_{X'}^\vee\mapright{ds'} 
\Omega_{Y'/Z'}|_{X'}]\lra \mathbb{L}_{X'/Z'}.\eeq
We also have the absolute perfect obstruction theory
\beq\label{g3} \phi_{X'}:
\mathbb{E}_{X'}=[E|_{X'}^\vee\mapright{ds'} \Omega_{Y'}|_{X'}]\lra \mathbb{L}_{X'}.
\eeq
As $\sigma'\circ s'=0$ on $Y'$, $\sigma'$ desends to a cosection
\beq\label{g5}
\sigma_{X'}:Ob_{X'}=\mathrm{coker}(ds':T_{Y'}|_{X'}\to E'|_{X'})\lra \sO_{X'}.
\eeq
Therefore we can apply the cosection localization principle.

\begin{theo}\label{g6} \cite[Theorem 5.1]{KLc}
Let $X'$ be a \DM stack equipped with a perfect obstruction theory $\phi_{X'}$ and a cosection $\sigma_{X'}:Ob_{X'}\to \sO_{X'}$. Then $X'$ admits a
localized virtual cycle $$[\![ X' ]\!]\virt\lloc\in A_*(S')$$ 
where $S'$ is the zero locus of $\sigma_{X'}$. Its image in $A_*(X')$ by the inclusion $S'\subset X'$ is the ordinary 
virtual fundamental class $[\![X']\!]\virt$ and $[\![X']\!]\virt\lloc$ is deformation invariant in the sense of intersection theory (cf. \cite{BeFa}).
\end{theo}

Under Assumption \ref{aa}, the construction of $[\![X']\!]\virt\lloc$ goes as follows: 
By \cite[Proposition 4.3]{KLc}, the normal cone $C_{X'/Y'}\subset E'|_{X'}$ has support in
$$E'|_{X'}(\sigma_{X'})=E'|_{S'}\cup \mathrm{ker}(\sigma_{X'}:E'|_{X'-S'}\to \sO_{X'-S'}).$$
Then we apply the cosection localized Gysin map 
\beq\label{g8}
0^!_{E'|_{X'},\sigma_{X'}}:A_*(E'|_{X'}(\sigma_{X'}))\to A_*(S')\eeq
to the cycle $[\![C_{X'/Y'}]\!]$ to obtain
\beq\label{g7} 
[\![X']\!]\virt\lloc=0^!_{E|_{X'},\sigma_{X'}}[\![C_{X'/Y'}]\!]\in A_*(S').
\eeq
Since $\dim C_{X'/Y'}=\dim Y=\dim S+\mathrm{rank}(M)-1$, the dimension of $[\![X']\!]\virt\lloc$
is \beq\label{g21}\dim S+\mathrm{rank}(M)-1-\mathrm{rank}(E)=3g-3+n-1+\sum_i\chi(L_i).\eeq

This class $[\![X']\!]\virt\lloc$ depends only on the perfect obstruction theory $\phi_{X'}$ and the cosection $\sigma_{X'}$.
In particular, the virtual cycle $[\![X']\!]\virt\lloc\in A_*(S')$ is independent of a choice of the resolution \eqref{d34}, and hence the choices of $Y$, $E$, etc in \S\ref{S5.3}. 

The construction of \eqref{g8} in \cite[\S2]{KLc} under Assumption \ref{aa} goes as follows: 
Let $\rho:\tY'\to Y'$ be the blowup of $Y'$ along $\fZ'$ (equivalently along the ideal $\sigma'(E')$) so that the pullback of $\sigma'$ is a surjection 
$\tE'=\rho^*E'\to \sO_{\tY'}(-\tfZ')$ where $\tfZ'$ denotes the exceptional divisor of $\rho$. 
Restricting these to $\tX'=X'\times_{Y'}\tY'$, we have a short exact sequence
$$0\lra F'\lra \tE'|_{\tX'}\lra \sO_{\tX'}(-\tS')\lra 0$$
of locally free sheaves where $\tS'=S'\times_{Y'}\tY'$. 
For $\xi'\in A_*(E'|_{X'}(\sigma_{X'}))$, we pick $\zeta'\in A_*(F')$ and $\eta'\in A_*(E'|_{S'})$  
such that 
\beq\label{g9} \xi'=\rho_*\zeta'+{\imath'}_*\eta',\and \xi'|_{X'-S'}=\zeta'|_{\tX'-\tS'}\eeq
where ${\imath'}_*$ is the pushforward induced by the inclusion ${S'}\subset X'$. 
Then \eqref{g8} is defined by
\beq\label{g10} 
0^!_{E'|_{X'},\sigma_{X'}}(\xi')= -{\rho_S}_*\left( \tS'\cdot 0^!_{F'}(\zeta')\right)+0^!_{E'|_{S'}}(\eta')\in A_*(S')
\eeq
where $\rho_S:\tS'\to S'$ is the restriction of $\rho$, $0^!_{F'}$ and $0^!_{E'|_{S'}}$ denote the ordinary Gysin maps and $\tS'\cdot $ denotes the intersection with the Cartier divisor $\tS'$ (cf. \cite{Fulton}).
By \cite[\S2]{KLc}, \eqref{g10} is independent of all the choices. Moreover, instead of $\rho$, 
we may use any $\sigma_{X'}$-regularizing morphism (Definition \ref{a18}) and 
\eqref{g8} is independent of this choice as well.

\subsection{Integral transformations by virtual cycles}\label{S2.3}

In this subsection, we define integral transformations by the virtual cycles constructed in Theorem \ref{g6}.
We will see in the subsequent section that these transformations form cohomological field theories. 

Under Assumption \ref{aa}, by Theorem \ref{g6}, we have the virtual cycle $[\![X']\!]\virt\lloc\in A_*(S')$
where $S'=S\times_XX'$. 
From \eqref{g11}, we have morphisms
\beq\label{g15} Z'\mapleft{q} S'\mapright{p} S\eeq
where $p$ is obtained from $g$ by base change and $q$ is the restriction of $\bq'$ to $S'$. 
Since $g$ is the blowup at the origin, $p$ is proper. As $S=S\urig_{g,\ugamma}$ is proper, so is $S'$.

As the blowup $Z'$ is a line bundle over $\PP Z=\PP w_\ugamma^{-1}(0)$, 
the cohomology $H^*(Z')$ of $Z'$ is isomorphic to 
$H^*(\PP w_\ugamma^{-1}(0))\supset H^*_{\mathrm{prim}}(\PP w_\ugamma^{-1}(0)).$
By the Thom-Sebastiani isomorphism \eqref{d19}, we have
\begin{align}\label{g12}
\cH_\ugamma&=\bigotimes_{j=1}^n\cH_{\gamma_j}\cong H^{\sum_jN_{\gamma_j}-2}_{\mathrm{prim}}(\PP w_\ugamma^{-1}(0))^{G^n}\\
&\subset H^{\sum_jN_{\gamma_j}-2}(\PP w_\ugamma^{-1}(0))=H^{\sum_jN_{\gamma_j}-2}(Z').\nonumber
\end{align}
Since $N_{\gamma_j}$ are often an odd number, a class in $\cH_\ugamma$ is not algebraic in general. 
Hence for a Fourier-Mukai type integral transformation with our state space $\cH_\ugamma$, we cannot use the Chow groups. 

For an irreducible variety $V$, we can associate the Borel-Moore homology class of $V$ after choosing a suitable triangulation. See \cite{Iver, Bred} for Borel-Moore homology. We thus have the cycle class map (cf. \cite[Chapter 19]{Fulton})
\beq\label{g16} h_{S'}:A_*(S')\to H_*(S')\eeq 
and the homological virtual cycle
\beq\label{g17} [X']\virt\lloc=h_{S'}[\![X']\!]\virt\lloc\in H_*(S').\eeq

By \eqref{g15} and \eqref{g17}, we define our integral transformations as  
\beq\label{g18}\Phi_{[X']\virt\lloc}: H^*(Z')\to H^*(S),\quad \Phi_{[X']\virt\lloc}(v)=p_*( [X']\virt\lloc\cap q^*(v)).\eeq

We have a forgetful morphism 
$$\mathrm{st}:S=S_{g,\ugamma}\urig\lra \rmgn, \quad (C,p_j,L_i,\varphi,\psi)\mapsto (|C|,|p_j|)$$
whose pushforward is denoted by $\mathrm{st}_*:H_*(S)\to H_*(\rmgn)\cong H^*(\rmgn)$.
Composing \eqref{g18} and \eqref{g12} with 
\beq\label{g20} \frac{(-1)^D}{\deg \, \mathrm{st}} \mathrm{st}_*:H_*(S)\lra H_*(\overline{M}_{g,n})\eeq 
where $D=-\sum_i\chi(L_i)$, we obtain the composite
\beq\label{ga}
 \Omega'_{g,n,\ugamma}:\cH_\ugamma \lra H^*(Z')\lra H_*(S)\lra H_*(\rmgn)\cong H^*(\rmgn).
\eeq
Summing up for $\ugamma\in G^n$, we obtain

\begin{defi}\label{g19} For $g\ge 0$ and $n\ge 0$ with $2g-2+n>0$, we have homomorphisms
\beq\label{g26} \Omega'_{g,n}:\cH^{\otimes n}=\bigoplus_\ugamma\cH_\ugamma \lra H_*(\rmgn)\cong H^*(\rmgn). 
\eeq
\end{defi}
By \eqref{g21}, the image of $\Omega'_{g,n}|_{\cH_\ugamma}$ lies in degree
\beq\label{g22}
2\Bigl( 3g-3+n-1+\sum_i\chi(L_i)\Bigr)-\Bigl( \sum_j N_j-2\Bigr)=\eeq
$$\qquad\qquad=6g-6+2n+2\sum_i\chi(L_i)-\sum_j N_{\gamma_j}$$
which matches the computation in \cite[Theorem 4.1.1]{FJR}.

In \S\ref{S4}, we will see that
the homomorphisms in Definition \ref{g19} form a cohomological field theoy by comparing them 
with the cohomological field theory constructed in \cite{KLq}.

\section{Quantum singularity theories via intersection homology}\label{S3}

In this section, we recall the construction of cohomological field theories by intersection homology in \cite{KLq}.

We refer to \cite[\S2]{KLq} for useful facts about Borel-Moore homology and intersection homology.
For instance, we will use the natural map
$$\epsilon_Y:\ih_i(Y)\to H_i(Y)$$
that sends the middle perversity intersection homology cycles to itself in the Borel-Moore homology. 
Also, we will use proper pushforwards and placid (flat) pullbacks of Borel-Moore homology groups. 

We first recall the cosection localized Gysin maps for intersection homology groups. 
Let $E$ be a vector bundle of rank $r$ over a \DM stack $Y$. Let $s$ be a section and $X=s^{-1}(0)$.
The canonical orientation on the fibers by the complex structure gives us the Thom class $\tau_{Y/E}\in H^{2r}(E,E-0_E)=H^{2r}_Y(E)$ where $0_E$ denotes the zero section of $E$. 
The section $s$ induces a map $(Y,X)\to (E,0_E)$ and $e(E,s)=s^*\tau_{Y/E}\in H^{2r}(Y,Y-X)=H^{2r}_X(Y)$. The (ordinary) Gysin map is now defined as
$$s^!:H_i(Y)\lra H_{i-2r}(X),\quad \xi\mapsto \xi\cap e(E,s).$$

When $E$ is equipped with a cosection $\sigma:E\to \sO_Y$, $s^!$ further localizes to $S=X\cap \fZ$ 
where $\fZ$ is the locus where $\sigma$ is not surjective. Let $\imath:S\to X$ be the inclusion map.
The following cosection localized Gysin map is the main machinery in this section. 

\begin{theo}\label{a16} \cite[Theorem 3.2]{KLq}
Under Assumption \ref{aa}, we have a homomorphism 
\beq\label{a17} s^!_\sigma:\ih_i (Y)\lra H_{i-2r}(S)\eeq
whose composition with $\imath_*:H_{i-2r}(S)\to H_{i-2r}(X)$ equals $s^!\circ \epsilon_Y$.
\end{theo}

Here is an outline of the construction: For $s^!_\sigma$, we have to resolve the degeneracy of $\sigma$.
\begin{defi}\label{a18}
Under Assumption \ref{aa}, a proper morphism $\rho:\tY\to Y$ is called $\sigma$-\emph{regularizing} if it is an isomorphism over $Y-\fZ$ and the pullback $\tE=\rho^*E\to \sO_{\tY}$ of $\sigma:E\to \sO_Y$ factors through a surjective homomorphism $\tilde{\sigma}:\tE\to  \sO_{\tY}(-\tfZ)$ for an effective Cartier divisor $\tfZ$ lying over $\fZ$. 
\end{defi}
For instance, the blowup of $Y$ along the ideal $I=\sigma(E)\subset \sO_Y$ is $\sigma$-regularizing.

Let $\rho:\tY\to Y$ be a $\sigma$-regularizing morphism and $F$ be the kernel of the surjection $\tsig:\tE\to \sO_{\tY}(-\tfZ)$. 
By the decomposition theorem \cite{BBD}, for any $\xi\in H_i(Y)$, we can always find $\zeta\in H_i(\tY)$ and $\eta\in H_i(\fZ)$ such that 
$$\epsilon_Y(\xi)=\rho_*(\zeta)+\jmath_*(\eta), \quad \epsilon_Y(\xi)|_{Y-\fZ}=\zeta|_{\tY-\tfZ}$$
where $\jmath:\fZ\to Y$ denotes the inclusion (cf. \cite[Lemma 2.2]{KLq}). Then $s^!_\sigma$ is defined by 
\beq\label{a19} s^!_\sigma(\xi)=
-{\rho_{S}}_*\left(\zeta\cap e(F,\tilde{s})\cap e(\sO_{\widetilde{X}}(\tS), t_{\tS})\right) + \eta\cap e(E|_\fZ,s|_\fZ)\eeq
where $t_{\tS}$ is the section of $\sO_{\widetilde{X}}(\tS)$ whose zero locus is $\tS$. 
It was proved in \cite[\S3]{KLq} that $s^!_\sigma(\xi)$ is independent of the choices of $\zeta$, $\eta$ 
and $\rho$.

\bigskip

In the FJRW theory, with $S=S\urig_{g,\ugamma}$, $X=X\urig_{g,\ugamma}$, $B_\ugamma=\prod_{j=1}^n\CC^{N_{\gamma_j}}$ and $\underline{w}={w}_\ugamma$, Assumption \ref{aa} is satisfied and hence we have the cosection localized Gysin map 
\beq\label{a24} s_\sigma^!:\ih_i(Y)\lra H_{i-2\mathrm{rank}(E)}(S)\eeq
by Theorem \ref{a16}. 
Moreover since $\bq_Y$ is smooth, we have the pullback homomorphism
\beq\label{a25} \bq_Y^*:\ih_i(Z)\lra \ih_{i+\dim_\RR M-\dim_\RR B}(Y).\eeq
Composing \eqref{a24} and \eqref{a25} with the Thom-Sebastiani isomorphism 
$$\cH_\ugamma=\bigotimes_j\cH_{\gamma_j}\cong \ih_{\sum_j N_{\gamma_j}}({w}_\ugamma ^{-1}(0))^{G^n},$$ 
we obtain 
\beq\label{a23} \cH_\ugamma\cong \ih_{\sum_j N_{\gamma_j}}({w}_\ugamma ^{-1}(0))^{G^n}\mapright{\bq_Y^*} \ih_*(Y)\mapright{s^!_\sigma} H_{2\mathrm{vd}(X)-\sum N_{\gamma_j}}(S_{g,\ugamma}\urig)\eeq
where $\mathrm{vd}(X)=3g-3+n+\sum_i\chi(L_i)$ is the virtual dimension of $X$.

Composing \eqref{a23} with \eqref{g20} and summing over $\ugamma$,
we obtain
\beq\label{a22}
\Omega_{g,n}:\cH^{\otimes n}\lra H^*(\rmgn).\eeq

\begin{theo}\label{a20} \cite[Theorem 4.5]{KLq} The homomorphisms $\{\Omega_{g,n}\}_{2g-2+n>0}$ in \eqref{a22} define a cohomological field theory with a unit for the state space $\cH$.
Moreover this cohomological field theory coincides with that in \cite[Theorem 4.2.2]{FJR} when we restrict $\cH$ to the narrow sector  $\bigoplus_{\ugamma: N_{\gamma_j}=0,\, \forall j}\bigotimes_j\cH_{\gamma_j}$. 
\end{theo}

Here a cohomological field theory is a term that codifies nice properties expected from curve counting invariants as follows.
\begin{defi} Let $\cH$ be a vector space equipped with a basis $\{e_1,\cdots,e_m\}$ and a perfect pairing $\langle e_k,e_l\rangle=c_{kl}$. 
A \emph{cohomological field theory} with a unit $\mathbf{1}$ for the state space $\cH$ consists of homomorphisms
\beq\label{d36} \Omega_{g,n}:\cH^{\otimes n}\lra H^*(\rmgn), \quad \text{for }2g-2+n>0\eeq
satisfying the following:
\begin{enumerate}
\item if we let the symmetric group $S_n$ act on $\rmgn$ by permuting the markings and on $\cH^{\otimes n}$ by permuting the factors, $\Omega_{g,n}$ is $S_n$-equivariant; 
\item if we let $u:\overline{M}_{g-1,n+2}\to \overline{M}_{g,n}$ denote the gluing of the last two markings, then we have
\beq\label{d37} u^*\Omega_{g,n}(v_1,\cdots,v_n)=\sum_{k,l} c^{kl}\Omega_{g-1,n+2}(v_1,\cdots,v_n, e_k,e_l)\eeq 
in $H^*(\overline{M}_{g-1,n+2})$ for all $v_i\in \cH$ where $(c^{kl})=(c_{kl})^{-1}$;
\item if we let $u:\overline{M}_{g_1,n_1+1}\times \overline{M}_{g_2,n_2+1}\to \overline{M}_{g,n}$ with $g=g_1+g_2$ and $n=n_1+n_2$ denote the gluing of the last markings, then
\small 
\beq\label{d38} u^*\Omega_{g,n}(v_1,\cdots,v_n)=\sum_{k,l} c^{kl}\Omega_{g_1,n_1+1}(v_1,\cdots,v_{n_1}, e_k)\otimes  \Omega_{g_2,n_2+1}(v_{n_1+1},\cdots,v_{n}, e_l)\eeq 
\normalsize
in $H^*(\overline{M}_{g_1,n_1+1})\otimes H^*(\overline{M}_{g_2,n_2+1})$ for all $v_i\in \cH$;
\item if we let $\theta:\overline{M}_{g,n+1}\to \rmgn$ denote the morphism forgetting the last marking, we have 
\beq\label{d39} \Omega_{g,n+1}(v_1,\cdots,v_n,\mathbf{1})=\theta^*\Omega_{g,n}(v_1,\cdots,v_n),\quad \forall v_i\in \cH ;\eeq
\item $\Omega_{0,3}(v_1,v_2,\mathbf{1})=\langle v_1,v_2\rangle$ for $v_i\in \cH$. 
\end{enumerate}\end{defi}
For a smooth projective variety $Q$, letting $\rmgn(Q,d)$ denote the moduli stack of 
stable maps to $Q$ of genus $g$ and degree $d$, it is well known that
$$H^*(Q)^{\otimes n}\lra H^*(\rmgn),\quad (v_j)\mapsto p_*\left([\rmgn(Q,d)]\virt\cap \prod_{j=1}^nev_j^*(v_j)\right) $$
form a cohomological field theory where $p:\rmgn(Q,d)\to \rmgn$ is the forgetful morphism and $ev_j:\rmgn(Q,d)\to Q$ is the evaluation map at the $j$-th marking. 

\bigskip

The homomorphisms \eqref{a22} are suitable for proving the axioms of cohomological field theories while those in Definition \ref{g19} are defined by Fourier-Mukai type integral operators with \emph{algebraic kernels} $[\![X']\!]\virt\lloc$. In the subsequent section, we will prove that actually they are the same.

\bigskip

We end this section with the following.

\begin{prop}\label{g23}
Let $E$ be an algebraic vector bundle of rank $r$ over a smooth \DM stack $Y$.
Let $s\in H^0(E)$ and $\sigma\in H^0(E^\vee)$ satisfy $\sigma\circ s=0$. 
Let $X=\zero(s)$, $\fZ=\zero(\sigma)$ and $S=X\cap \fZ$. Then
$$s^!_\sigma[Y]=[X]\virt\lloc:=h_Y[\![X]\!]\virt\lloc \in H_{2\dim Y-2r}(S),
$$
where $s^!_\sigma$ is from Theorem \ref{a16} and $[\![X]\!]\virt\lloc$ is from Theorem \ref{g6} (with primes removed). 
\end{prop}
\begin{proof}
Suppose $\sigma=0$. Let $\Gamma\subset E\times \CC^*$ be the graph 
of the section $(y,t)\mapsto t^{-1}s(y)$. Let $\overline{\Gamma}$ be the closure of $\Gamma$ in $E\times \CC$.  
Then for $t\ne 0$, the fiber $\overline{\Gamma}_t$ over $t$ is isomorphic to $Y$ and the fiber $\overline{\Gamma}_0$ over $t=0$ is the normal cone $C_{X/Y}$ (cf. \cite{Fulton}).
As $Y=\overline{\Gamma}_1$ is homologous to $\overline{\Gamma}_0=C_{X/Y}$, 
$s^![Y]=s^![C_{X/Y}]$. The proposition now follows from the fact that $s^![C_{X/Y}]=h_X\circ 0^!_E[\![C_{X/Y}]\!]$ by \cite[Chapter 19]{Fulton} where $0^!_{E}$ denotes the algebraic Gysin map of Chow groups.

When  $\sigma$ is not necessarily zero, we let $\rho:\tY\to Y$ be the blowup of $Y$ along $\fZ$ so that we have an exact sequence
$$0\lra F\lra \tE\mapright{\tilde{\sigma}} \sO_{\tY}(-\tfZ)\lra 0$$
where $\tfZ$ is the exceptional divisor, $\tE$, $\tilde{s}$ and $\tilde{\sigma}$ denote the pullbacks of $E$, $s$ and $\sigma$ to $\tY$ respectively. Then 
$$[\![Y]\!]=\rho_*[\![Y']\!].$$ 
From the commutative diagram
\[\xymatrix{
\tY\times \CC^*\ar[d] \ar[r]^{t^{-1}\tilde{s}} & \tE\times\CC^*\ar[d]\\
Y\times\CC^*\ar[r]^{t^{-1}s} & E\times\CC^*,
}\]
we find that $[\![C_{X/Y}]\!]=\rho_*[\![C_{\tX/\tY}]\!]$ 
as $C_{X/Y}$ (resp. $C_{\tX/\tY}$) is rationally equivalent to $Y$ in $E$ (resp. $\tY$ in $\tE$). 
By \eqref{g10}, 
$$[\![X]\!]\virt\lloc=0^!_{E|_X,\sigma}[\![C_{X/Y}]\!]=-{\rho_S}_*(\tS\cdot 0^!_F[\![C_{\tX/\tY}]\!]).
$$ 
By \cite[Chapter 19]{Fulton}, \eqref{a19} and the case for trivial cosection, 
if we apply the cycle class map $h_S$, we obtain
$$h_S[\![X]\!]\virt\lloc=-{\rho_S}_*(t_{\tS}^! \tilde{s}^![\tY]) 
=s^!_\sigma[Y]$$
as desired. 
\end{proof}

A direct consequence of Proposition \ref{g23} is the following. 
\begin{coro}\label{g24}
Under Assumption \ref{aa}, letting $Z'\to Z=Z_\ugamma$ be the blowup at the origin and using the notation of \S\ref{S2}, we have
\beq\label{g25} [X']\virt\lloc={s'}^!_{\sigma'}[Y'], \and \Phi_{[X']\virt\lloc}(\alpha)=p_*({s'}^!_{\sigma'}[Y']\cap q^*(\alpha))\eeq 
for $\alpha\in H^*(Z')$. 
\end{coro}

\section{Comparison}\label{S4}
The goal of this section is to prove the following. 
\begin{theo}\label{g30} 
The homomorphism $\Omega_{g,n}$ in \eqref{a22} equals $\Omega'_{g,n}$  in \eqref{g26}.
Hence, the integral transformations $\Phi_{[X']\virt\lloc}$ with algebraic kernels 
$[X']\virt\lloc$ give rise to cohomological field theories $\{\Omega'_{g,n}\}$.
\end{theo}
\begin{proof}
Recall that the maps \eqref{a22} (resp. \eqref{g26}) are obtained by composing $s^!_\sigma\bq^*$ (resp. $\Phi_{[X']\virt\lloc}$) with the stabilization \eqref{g20}. 
Therefore the theorem follows once we show that whenever 
\beq\label{g34} \epsilon_Z(v)=g_*([Z']\cap \alpha)\quad \text{for }\alpha\in H^*(Z'), v\in \ih_*(Z'),\eeq 
\beq\label{ab} \Phi_{[X']\virt\lloc}(\alpha)=p_*([X']\virt\lloc\cap q^*\alpha)=s^!_\sigma\bq^*(v), \eeq 
which is Theorem \ref{a0} below. Note that since $Z=w_\ugamma^{-1}(0)$ is a nondegenerate hypersurface singularity, if we let $m-1=\dim Z$, $\ih_m(Z)\cong H^{m-2}_{\mathrm{prim}}(\PP Z)\subset H^{m-2}(\PP Z)\cong H^{m-2}(Z')\cong H_m(Z')$ and we have a commutative diagram
\beq\label{g31}\xymatrix{
H_m(Z')\ar[d]_{g_*} & H^{m-2}(Z')\ar[l]_{[Z']\cap }^{\cong}& H^{m-2}(\PP Z)\ar[l]_\cong\\
H_m(Z) &\ih_m(Z)\ar[l]_{\epsilon_Z}\ar[r]^\cong &H^{m-2}_{\mathrm{prim}}(\PP Z).\ar@{^(->}[u]
}\eeq
By the isomorphism $\ih_m(Z)\cong H^{m-2}_{\mathrm{prim}}(\PP Z)$, $\alpha\in H^{m-2}_{\mathrm{prim}}(\PP Z)$ and $v\in \ih_m(Z)$ determine each other uniquely.
\end{proof}

For a proof of \eqref{ab}, we need a couple of propositions on the cosection localized Gysin map $s^!_\sigma$.

By \cite[5.2]{GM83}, for a \DM stack $Y$, we have a cap product
\beq\label{h2} \ih_i(Y)\times H^j(Y)\mapright{\cap} \ih_{i-j}(Y),\quad (\xi,\alpha)\mapsto \xi\cap \alpha\eeq
which fits into a commutative diagram
\beq\label{h3} \xymatrix{
\ih_i(Y)\times H^j(Y)\ar[r]^{\cap} \ar[d]_{\epsilon_Y\times 1} &\ih_{i-j}(Y)\ar[d]^{\epsilon_Y}\\
H_i(Y)\times H^j(Y)\ar[r]^{\cap} & \ih_{i-j}(Y)
}\eeq
where the bottom arrow is the usual cap product \cite[IX.3]{Iver}. The cap product satisfies the projection formula (cf. \cite[IX.3.7]{Iver})
\beq\label{h4} g_*(\xi\cap f^*\alpha)=f_*\xi\cap \alpha, \quad \xi\in H_i(X), \ \alpha\in H^j_W(Y)=H^j(Y,Y-W) 
\eeq
where $f:X\to Y$ is proper and $g:f^{-1}(W)\to W$ is the restriction of $f$ to $f^{-1}(W)=W\times_YX$ for closed $W\subset Y$. For closed $A$ and $B$ in $Y$ , 
\beq\label{h5} (\xi\cap \alpha)\cap \beta|_A=\xi\cap (\alpha\cup \beta)=(\xi\cap \beta)\cap \alpha|_B\eeq
for $\xi\in H_i(Y)$, $\alpha\in H^j_A(Y)$ and $\beta\in H^k_B(Y).$

\begin{prop}\label{h1} Let $X=\zero(s)$, $\fZ=\zero(\sigma)$ and $S=X\cap \fZ$.  
For $\alpha\in H^j(Y)$ and $\xi\in \ih_i(Y)$, we have
\[ s^!_\sigma(\xi\cap \alpha)=s^!_\sigma(\xi)\cap \alpha|_S \in H_{i-j-2r}(S).\]
\end{prop}
\begin{proof}
Let $\rho:\tY\to Y$ be a $\sigma$-regularizing birational morphism such that the exceptional divisor $\tfZ$ lies over $\fZ=\zero(\sigma)$. The cosection $\sigma:E\to \sO_Y$ lifts to a surjective homomorphism $\tilde{\sigma}:\tE=\rho^*E\to \sO_{\tY}(-\tfZ)$ whose kernel is denoted by $F$.
For $s^!_\sigma(\xi)$, we pick 
$\zeta\in H_i(\tY)$ and $\eta\in H_i(\fZ)$ such that
$$\epsilon_Y(\xi)=\rho_*\zeta+\jmath_*\eta,\quad \epsilon_Y(\xi)|_{Y-\fZ}=\zeta|_{\tY-\tfZ}$$
where $\jmath:\fZ\to Y$ denotes the inclusion. 
By the definition of $s^!_\sigma$, we have 
\beq\label{h6} s^!_\sigma(\xi)=\eta\cap e(E|_\fZ,s|_\fZ)-{\rho_S}_*(\zeta\cap e(F,\tilde{s})\cap e(\sO_{\tX}(\tS),t_{\tS}) )\eeq
where $\rho_S:\tS\to S$ is the restriction of $\rho$ to $\tS=\tY\times_YS$ and $t_{\tS}$ is the section of 
$\sO_{\tX}(\tS)$ whose vanishing locus is the divisor $\tS$.

By \eqref{h3} and \eqref{h4}, we have
$$\epsilon_Y(\xi\cap \alpha)=\epsilon_Y(\xi)\cap \alpha=\rho_*(\zeta)\cap \alpha+\jmath_*(\eta)\cap \alpha$$
$$=\rho_*(\zeta \cap \rho^*\alpha)+\jmath_*(\eta\cap \jmath^*\alpha).$$
Hence we have
$$
s^!_\sigma(\xi\cap \alpha)=(\eta\cap \jmath^*\alpha)\cap e(E|_\fZ,s|_\fZ)-{\rho_S}_*((\zeta\cap \rho^*\alpha)\cap e(F,\tilde{s})\cap e(\sO_{\tX}(\tS),t_{\tS}) ).
$$
By \eqref{h5}, the above line equals
\beq\label{h8}
(\eta\cap e(E|_\fZ,s|_\fZ))\cap \alpha|_S-{\rho_S}_*(\zeta\cap e(F,\tilde{s})\cap e(\sO_{\tX}(\tS),t_{\tS})\cap \rho_S^*\alpha|_S ).
\eeq
By the projection formula again and \eqref{h6}, \eqref{h8} equals
\[
(\eta\cap e(E|_\fZ,s|_\fZ))\cap \alpha|_S-{\rho_S}_*(\zeta\cap e(F,\tilde{s})\cap e(\sO_{\tX}(\tS),t_{\tS}))\cap \alpha|_S =s^!_\sigma(\xi)\cap \alpha|_S.
\]
This proves the proposition.
\end{proof}

\begin{prop}\label{h9}
Let $f:Y'\to Y$ be a proper morphism of normal \DM stacks. Let $E', s', \sigma', X', \fZ', S'$ etc be the pullbacks of $E, s, \sigma, X, \fZ, S$ etc by $f$. Let $p:S'\to S$ denote the restriction of $f$ to $S'$.   
Let $\xi\in \ih_i(Y)$. Suppose there exists $\xi'\in \ih_i(Y')$ such that $f_*\epsilon_{Y'}(\xi')=\epsilon_Y(\xi).$ Then we have
$$p_*{s'}^!_{\sigma'}(\xi')=s^!_\sigma(\xi)\in H_{i-2r}(S).$$
\end{prop}
\begin{proof}
Let $\rho:\tY\to Y$ be a $\sigma$-regularizing birational morphism so that $\sigma$ lifts to the surjective homomorphism 
$\tilde{\sigma}:\tE=\rho^*E\to \sO_{\tY}(-\tfZ)$ over $\tY$ with kernel $F$.
Consider the fiber product
\[\xymatrix{
\tY'\ar[r]^{\tilde{f}} \ar[d]_{\rho'} & \tY\ar[d]^\rho\\
Y'\ar[r]^f & Y
}\]
Then $\tilde{\sigma}$ lifts to a surjective homomorphism
$\tilde{\sigma}':\tE'={\rho'}^*E'\to \sO_{\tY'}(-\tfZ')$ where $\tfZ'=\tfZ\times_{\tY}\tY'.$

For ${s'}^!_{\sigma'}(\xi')$, we pick $\zeta'\in H_i(\tY')$ and $\eta\in H_i(\fZ')$ such that
\beq\label{h10} \epsilon_{Y'}(\xi')=\rho'_*\zeta'+\jmath'_*\eta',\quad \epsilon_{Y'}(\xi')|_{Y'-\fZ'}=\zeta'|_{\tY'-\tfZ'}  \eeq
where $\jmath':\fZ'\to Y'$ denotes the inclusion. 
By the definition of ${s'}^!_{\sigma'}$, we have
\beq\label{h11}
{s'}^!_{\sigma'}(\xi')=\eta'\cap e(E'|_{\fZ'},s'|_{\fZ'})-{\rho_{S'}}_*(\zeta'\cap e(F',\tilde{s}')\cap e(\sO_{\tX'}(\tS'),t_{\tS'}) )
\eeq

By applying $f_*$ to \eqref{h10}, we have
$$\epsilon_Y(\xi)=f_*\epsilon_{Y'}(\xi')=f_*\rho'_*\zeta'+f_*\jmath'_*\eta' =\rho_*(\tilde{f}_*(\zeta'))+{\jmath}_*({f_{\fZ}}_*(\eta'))$$
where $f_\fZ:\fZ'\to \fZ$ denotes the restriction of $f$ to $\fZ'=\fZ\times_YY'$. 
Moreover, $\tilde{f}_*(\zeta')|_{\tY-\tfZ}=\tilde{f}_*\epsilon_{Y'}(\xi')|_{Y-\fZ}=\epsilon_Y(\xi)|_{Y-\fZ}$ by \eqref{h10}. Hence, we have
\beq\label{h12}
{s}^!_{\sigma}(\xi)={f_{\fZ}}_*(\eta')\cap e(E|_{\fZ},s|_{\fZ})-{\rho_{S}}_*(\tilde{f}_*(\zeta')\cap e(F,\tilde{s})\cap e(\sO_{\tX}(\tS),t_{\tS}) ).
\eeq
By the projection formula \eqref{h4}, ${f_{\fZ}}_*(\eta')\cap e(E|_{\fZ},s|_{\fZ})=p_*(\eta'\cap e(E'|_{\fZ'},s|_{\fZ'}))$ and 
\begin{align*}
{\rho_{S}}_*(\tilde{f}_*(\zeta')\cap & e(F,\tilde{s})\cap e(\sO_{\tX}(\tS),t_{\tS}) )\\
&={\rho_{S}}_*\tilde{p}_*(\zeta'\cap e(F',\tilde{s}')\cap e(\sO_{\tX'}(\tS'),t_{\tS'}) )\\
&=p_*{\rho_{S'}}_*(\zeta'\cap e(F',\tilde{s}')\cap e(\sO_{\tX'}(\tS'),t_{\tS'}) ),
\end{align*}
where $\tilde{p}$ comes from the fiber diagram
\[\xymatrix{
\tS'\ar[r]^{\tilde{p}} \ar[d]_{\rho_{S'}} & \tS\ar[d]^{\rho_S}\\
S'\ar[r]^p & S.}\]
Hence by \eqref{h11}, \eqref{h12} equals 
$$p_*\left(\eta'\cap e(E'|_{\fZ'},s|_{\fZ'})-{\rho_{S'}}_*(\zeta'\cap e(F',\tilde{s}')\cap e(\sO_{\tX'}(\tS'),t_{\tS'}) )\right)=p_*{s'}^!_{\sigma'}(\xi').$$
This proves the proposition.
\end{proof}

\begin{theo}\label{h0}
Under Assumption \ref{aa}, we let $g:Z'\to Z$ be a birational proper morphism with $Z'$ smooth. Consider the fiber product
\beq\label{a13}\xymatrix{ 
Y'\ar[r]^f\ar[d]_{\bq'} & Y\ar[d]^{\bq}\\
Z'\ar[r]^g & Z}\eeq
so that $Y'$ is smooth. Let $E', s', \sigma', X', S'$ etc be the pullbacks of $E, s, \sigma, X, S$ etc by $f$. 
Let $p:S'\to S$ denote the restriction of $f$ to $S'$ and $q:S'\to Z'$ be the restriction of $\bq'$ to $S'$. 
Then \eqref{ab} holds and we have a commutative diagram
\beq\label{a0}\xymatrix{
H_*(Z')\ar[d]_{g_*} & H^*(Z')\ar[l]_{[Z']\cap }^{\cong}\ar[r]^{q^*} & H^*(S')\ar[rr]^{[X']\virt\lloc\cap} &&H_*(S')\ar[d]^{p_*}\\
H_*(Z) &\ih_*(Z)\ar[l]_{\epsilon_Z}\ar[r]^{\bq^*} &\ih_*(Y)\ar[rr]^{s^!_{\sigma}} &&H_*(S)
}\eeq
where by Corollary \ref{g24}, the virtual cycle for $X'$ is 
\beq\label{ac} [X']\virt\lloc ={s'}^!_{\sigma'}[Y']=[Y']\cap e_{\sigma'}(E',s')\in H_*(S').\eeq 
\end{theo}
\begin{proof}
By \eqref{ac} and Proposition \ref{h1}, we have
\beq\label{a14}
p_*([X']\virt\lloc\cap q^*\alpha)=p_*({s'}^!_{\sigma'}[Y']\cap q^*\alpha)=p_*{s'}^!_{\sigma'}([Y']\cap {\bq'}^*\alpha)=p_*{s'}^!_{\sigma'}{\bq'}^*([Z']\cap \alpha)\eeq
since $Y', Z', \bq'$ are smooth. 
From the fiber diagram \eqref{a13}, we have 
$$f_*{\bq'}^*([Z']\cap \alpha)=\bq^*g_*([Z']\cap \alpha)=\bq^*\epsilon_Z(v)=\epsilon_Y(\bq^*v)$$
since $\bq$ is smooth. 
Therefore, by Proposition \ref{h9}, \eqref{a14} equals $s^!_\sigma\bq^*v$ as desired. 
\end{proof}


\bibliographystyle{amsplain}

\end{document}